\documentclass[10pt,a4paper]{article}
\usepackage[latin1]{inputenc}
\usepackage{amssymb,mathtools,amsthm,amsmath,amsfonts, latexsym,verbatim,geometry,fontenc}
\newcommand{\Z}{\mathbb{Z}}

\usepackage[dvipsnames]{xcolor} 
\usepackage{geometry}

\date{}
\newtheorem{theorem}{Theorem}[section]
\newtheorem{corollary}{Corollary}[theorem]
\newtheorem{lemma}[theorem]{Lemma}

\newtheorem{definition}{Definition}
\title{On signs of eigenvalues of modular forms satisfying Ramanujan conjecture}
\author{Nagarjuna Chary Addanki}
\begin{document} 

\newcommand{\ds}[1]{\displaystyle{#1}}
\newcommand{\G}{\mathrm{GSp}}
\newcommand{\Sp}{\mathrm{Sp}}
\newcommand{\GL}{\mathrm{GL}}
\newcommand{\SL}{\mathrm{SL}}
\newcommand{\mf}[1]{\mathfrak{#1}}
\newcommand{\mb}[1]{\mathbb{#1}} 
\newcommand{\mr}[1]{\mathrm{#1}}
\newcommand{\M}{\mr{M}}
\maketitle
\section*{Introduction} 
Siegel modular forms of genus $n$ and weight $k$ of level $N$ are holomorphic functions on the Siegel upper half space $\mathbb{H}_n$ that satisfy the modularity condition with respect to congruence subgroups of $\Sp_{2n}(\mb{Q})$. We denote a congruence subgroup of genus $n$ and level $N$ by $\Gamma^{(n)}(N)$. Let $\M_k(\Gamma^{(n)}(N))$ denote the space of Siegel modular forms of weight $k$, genus $n$ over $\Gamma^{(n)}(N)$ and $\mr{S}_k(\Gamma^{(n)}(N))$ denote the subspace of cuspidal forms. The space of cusp forms has a special basis called Hecke eigenforms. They arise as eigenvectors with respect to operators called the Hecke operators. For each positive integer $m$ there is a Hecke operator associated to it, denoted by $T(m)$. For a Hecke eigenform $F$, let $\lambda_F(m)$ denote the eigenvalue of $T(m)$. For a normalised eigenform these eigenvalues are real. Hence the behavior of signs of the eigenvalues can be studied.

\cite[Theorem $5$]{MR2726725} proved that for two normalized Hecke eigenforms $F \in \mr{S}_{k_1}(\Gamma^{(1)}(N_1))$ and $G \in \mr{S}_{k_2}(\Gamma^{(1)}(N_2))$, if $\mr{sign}(\lambda_F(p^r))= \mr{sign}(\lambda_G(p^r))$ for almost all $p$ and $r$ then $F=G$. Thus two genus 1 modular forms can be compared by studying the signs of the eigenvalues. In case of genus $2$, the space $\mr{S}_k(\Gamma^{(2)}(1))$ decomposes into two subspaces, mutually orthogonal to each other. The first subspace is known as the Maass subspace and it is generated by Saito-Kurokawa lifts. Saito-Kurokawa lifts are modular forms of genus $2$ constructed using a form of genus $1$ as explained in \cite{MR0549401}. Breulmann, in \cite{MR1719682}, showed that $F \in S_k(\Gamma^{(2)}(1))$ is a Saito-Kurokawa lift if and only if $\lambda_F(m) >0$ for all $ m \geq 1$. Kohnen, in \cite{MR2262899}, showed that a Hecke eigenform $ F \in S_k(\Gamma^{(2)}(1))$ is in the orthogonal complement of the Maass space if and only if there are infinitely many sign changes in the sequence $\{ \lambda_F(m) \}_{  m \geq 1}$. These results underscore the significance of analyzing the signs of Hecke eigenvalues. In this article, we focus on the eigenvalues of the modular forms of genus $2$ with level. Ikeda lifts, which are generalizations of the Saito-Kurokawa lifts to a higher genus, show a similar property. In \cite{addanki2024} we proved that for a genus $4$ Ikeda lift $F$, for a fixed $r$ $\lambda_F(p^r) \geq 0$ for all sufficiently large $p.$

Pitale and Schmidt in \cite{MR2425722} proved that, for a $F \in \mr{S}_k(\Gamma^{(2)}_0(N))$ and in the orthogonal compliment of the Maass subspace, there are infinitely many prime numbers $p$ such that the sequence of Hecke eigenvalues $\{ \lambda_F(p^r) \}_{r \geq 1} $ has infinitely many sign changes. Theorem $4$ of \cite{MR4198744} proves that, under a specific condition, if $F \in S_{k_1}(\Gamma^{(2)}(1))$ and $G \in S_{k_2}(\Gamma^{(2)}(1))$ are in orthogonal complement of their respective Maass subspaces then for a set of primes of positive density, $\lambda_F(p)\lambda_G(p) <0$. In this article, we use the techniques used in \cite{MR4198744} to prove a similar result for Siegel modular forms with level that satisfy the Ramanujan conjecture. The main result is

\begin{theorem}
Let $F \in \mr{S}_{k_1}(\Gamma^{(2)}(N_1))$ and $G \in \mr{S}_{k_2}(\Gamma^{(2)}(N_2))$ be two Hecke eigenforms that satisfy the Ramanujan conjecture. Let $\pi_F$ and $\pi_G$ be cuspidal automorphic representations of $\G_4(\mb{A_Q})$ associated with $F$ and $G$ respectively. Assume that if 
$$L(s,\pi_F,\mr{spin})=L(s,\pi_1)L(s,\pi_2) \ and \ L(s,\pi_G,\mr{spin})=L(s,\tau_1)L(s,\tau_2)$$ for some cuspidal automorphic representations $\pi_1, \pi_2, \tau_1$ and $\tau_2$ over $\GL_2(\mb{A_Q})$ then all representations are pairwise non isomorphic. Also, assume that for some $ c \in (0,4)$ and $\alpha > 15/16,$
$$ \# \{ p \leq x : \lvert \lambda_G(p) \rvert > c \} \geq \alpha \frac{x}{\log{x}} $$ for sufficiently large $x$. Then the set of primes $ \{p : \lambda_F(p)\lambda_G(p) < 0 \}$ has a positive density. 
\end{theorem}
The main result is on the signs of $\lambda_F(p)$. When the local factor of the spin L-function of a Hecke eigenform $F$ is written as a Dirichlet series, the coefficient of $p^{-s}$ is the eigenvalue $\lambda_F(p)$. Hence, to study the properties of $\lambda_F(p)$ it is sufficient to study the coefficient $p^{-s}$ of the L-function. We extensively use the prime number theorem stated as Theorem 3 of \cite{MR2364718} for asymptotic behavior of the coefficients.

\textbf{Outline of the paper:} In the first section of the article, we talk about basics of Siegel modular forms, automorphic representation associated with a modular form, and give a brief description of representations associated to the modular forms satisfying the Ramanujan conjecture. This section gives a description of the different types of L-functions to be expected for a eigenform satisfying the Ramanujan conjecture. In Section $2$, we show the relation between the eigenvalue $\lambda_F(p)$ and the coefficient of $p^{-s}$ of the L-function. Using Theorem $3$ of \cite{MR2364718}, we prove few technical results that would be used for the main result. In the final section, we prove the main result and explain the assumptions made in the theorem.

\section{Automorphic representations} 

For any ring $R$, let$$\G_{2n}(R)  =  \Big\{   g = \begin{pmatrix}
A & B\\
C & D 
\end{pmatrix} \in \GL_{2n}(R) : \prescript{t}{}{g} J g = \mu(g) J, \ J = \begin{pmatrix}
0 & 1_n\\
-1_n & 0 
\end{pmatrix}  \Big\} $$ where $\mu$ is the similitude homomorphism, $1_n$ is identity matrix of size $n$ and $A,B,C,D \in M_n(R)$. 
$$ \Sp_{2n}(R) \coloneqq \{ g \in \G_{2n}(R) :  \mu(g) = 1 \}  .$$ Let $N$ be a positive integer. Principal congruence subgroup of level $N$ and genus $n$ is defined to be the subgroup
$$ \{ g \in \Sp_{2n}(\mb{Z}) : g \equiv 1_{2n}(\mr{mod} \ N) \}.$$ Congruence subgroup of level $N$ and genus $n$ is a finite indexed subgroup of $\Sp_{2n}(\mb{Z})$ containing the principal congruence subgroup. 

Let $\Gamma^{(n)}(N)$ denote a congruence subgroup of level $N$ and genus $n$. A Siegel modular form F, of genus $n$, weight $k$ with respect to $\Gamma^{(n)}(N)$, is a holomorphic function on the Siegel upper half space
$$ \mathbb{H}_n \coloneqq \{ Z : Z \in M_n(\mathbb{C}), \ \prescript{t}{}{Z} =Z \ \text{and} \ \mathrm{Im}(Z) >0 \} $$ satisfying the following two conditions. \begin{enumerate}
    \item Modularity condition 
$$ F((AZ+B)(CZ+D)^{-1}) = \mathrm{det}(CZ+D)^k F(Z) \quad \forall \ \begin{pmatrix}
A & B\\
C & D
\end{pmatrix} \in \Gamma^{(n)}(N) \  \text{and} \  Z \in \mathbb{H}_n .$$
\item  For $n =1$,  $F(Z)$ is bounded on $\{ Z = X+i Y : Y \geq Y_0 \} \ \forall \ Y_0 > 0.$ \end{enumerate} Holomorphy and modularity imply that a Siegel modular form has a Fourier expansion of the form
$$ F(Z) = \sum_{\substack{ T= T^t, \ T \geq 0 \\ T \ \text{half integral}}} A(T) e^{ 2 \pi i tr(TZ)} .$$ Siegel modular forms over $\Gamma^{(n)}(N)$ are generally called Siegel modular forms with level. Let $M_k(\Gamma^{(n)}(N))$ denote the space of Siegel modular forms of genus $n$ and weight $k$ over $\Gamma^{(n)}(N)$. $F$ is called cuspidal if $A(T) = 0$ unless $T > 0$ and let $\mr{S}_k(\Gamma^{(n)}(N))$ denote the subspace of cusp forms. This article focuses on the cusp forms of the genus $2$ with level. 

In case of genus $2$ there are $4$ congruence subgroups. They are
\begin{enumerate}

    \item Borel congruence subgroup 
    $$ B(N) = \Sp_4(\mb{Z}) \cap \begin{bmatrix} 
    \mb{Z} & N\mb{Z} & \mb{Z} & \mb{Z}  \\
     \mb{Z} & \mb{Z} & \mb{Z} &  \mb{Z} \\
    N\mb{Z} & N\mb{Z} & \mb{Z} &  \mb{Z}    \\ 
    N\mb{Z} & N\mb{Z} & N\mb{Z} & \mb{Z}     \\  
  \end{bmatrix} \quad  $$

  \item Siegel congruence subgroups 
  $$\Gamma_0^{2}(N) = \Sp_4(\mb{Z}) \cap  \begin{bmatrix} 
    \mb{Z} & \mb{Z} & \mb{Z} & \mb{Z}  \\
     \mb{Z} & \mb{Z} & \mb{Z} &  \mb{Z} \\
    N\mb{Z} & N\mb{Z} & \mb{Z} &  \mb{Z}    \\ 
    N\mb{Z} & N\mb{Z} & \mb{Z} & \mb{Z}     \\  
  \end{bmatrix} $$
  \item Klingen congruence subgroup 
  $$Q(N) = \Sp_4(\mb{Z}) \cap  \begin{bmatrix} 
    \mb{Z} & N\mb{Z} & \mb{Z} & \mb{Z}  \\
     \mb{Z} & \mb{Z} & \mb{Z} &  \mb{Z} \\
    \mb{Z} & N\mb{Z} & \mb{Z} &  \mb{Z}    \\ 
    N\mb{Z} & N\mb{Z} & N\mb{Z} & \mb{Z}     \\  
  \end{bmatrix} $$
  \item Paramodular congruence subgroup 
  $$K(N) = \Sp_4(\mb{Q}) \cap  \begin{bmatrix} 
    \mb{Z} & N\mb{Z} & \mb{Z} & \mb{Z}  \\
     \mb{Z} & \mb{Z} & \mb{Z} &  N^{-1}\mb{Z} \\
    \mb{Z} & N\mb{Z} & \mb{Z} &  \mb{Z}    \\ 
    N\mb{Z} & N\mb{Z} & N\mb{Z} & \mb{Z}     \\  
  \end{bmatrix}$$
\end{enumerate}

Let $\Gamma^{(2)}(N)$ represent one of the four congruence subgroups above. For each $\Gamma^{(2)}(N)$ we can find an open compact subgroup $K_{\mathfrak{f}}$ of $\G_4(\mb{A_Q})$ such that $\Gamma^{(2)}(N) = \G_4(\mb{Q}) \cap \G_4(\mb{R})^+K_{\mathfrak{f}}.$ Here $\G_4(\mb{R})^+$ is a subgroup of $\G_4(\mb{R})$ consisting of matrices with positive similitude. In the case of the congruence subgroups of genus $2$, we describe the construction of $K_{\mf{f}}$ below.

For a fixed $N$, let $r_p$ denote a positive integer such that $p^{r_p} | N$ and $p^{r_p+1} \nmid N.$ 
\begin{enumerate} \item If $\Gamma^{(2)}(N) = B(N)$ then $K_\mf{f} = \prod_{ p | N} B_{\mf{p}}(p^{r_p}) \prod_{ p \nmid N} \G_4(\mb{Z}_p)$ where  
   $$ B_{\mf{p}}(p^{r_p}) = \Sp_4(\mb{Z}_p) \cap \begin{bmatrix} 
    \mb{Z}_p & p^{r_p}\mb{Z}_p & \mb{Z}_p & \mb{Z}_p  \\
     \mb{Z}_p & \mb{Z}_p & \mb{Z}_p &  \mb{Z}_p \\
    p^{r_p}\mb{Z}_p & p^{r_p}\mb{Z}_p & \mb{Z}_p &  \mb{Z}_p    \\ 
    p^{r_p}\mb{Z}_p & p^{r_p}\mb{Z}_p & p^{r_p}\mb{Z}_p & \mb{Z}_p     \\  
  \end{bmatrix} $$

  \item If $\Gamma^{(2)}(N) = \Gamma_0^{2}(N)$ then $K_\mf{f} = \prod_{ p | N} \Gamma^2_{0,\mf{p}}(p^{r_p}) \prod_{ p \nmid N} \G_4(\mb{Z}_p)$ where  
  $$\Gamma_{0,\mf{p}}^{2}(p^{r_p}) = \Sp_4(\mb{Z}_p) \cap  \begin{bmatrix} 
    \mb{Z}_p & \mb{Z}_p & \mb{Z}_p & \mb{Z}_p  \\
     \mb{Z}_p & \mb{Z}_p & \mb{Z}_p &  \mb{Z}_p \\
    p^{r_p}\mb{Z}_p & p^{r_p}\mb{Z}_p & \mb{Z}_p &  \mb{Z}_p    \\ 
    p^{r_p}\mb{Z}_p & p^{r_p}\mb{Z}_p & \mb{Z}_p & \mb{Z}_p     \\  
  \end{bmatrix} $$
  \item If $\Gamma^{(2)}(N) = Q(N)$ then $K_\mf{f} = \prod_{ p | N} Q_{\mf{p}}(p^{r_p}) \prod_{ p \nmid N} \G_4(\mb{Z}_p)$ where 
  $$Q_{\mf{p}}(p^{r_p}) = \Sp_4(\mb{Z}_p) \cap  \begin{bmatrix} 
    \mb{Z}_p & p^{r_p}\mb{Z}_p & \mb{Z}_p & \mb{Z}_p  \\
     \mb{Z}_p & \mb{Z}_p & \mb{Z}_p &  \mb{Z}_p \\
    \mb{Z}_p & p^{r_p}\mb{Z}_p & \mb{Z}_p &  \mb{Z}_p    \\ 
    p^{r_p}\mb{Z}_p & p^{r_p}\mb{Z}_p & p^{r_p}\mb{Z}_p & \mb{Z}_p     \\  
  \end{bmatrix} $$
  \item  If $\Gamma^{(2)}(N) =K(N)$ then $K_\mf{f} = \prod_{ p | N} K_{\mf{p}}(p^{r_p}) \prod_{ p \nmid N} \G_4(\mb{Z}_p)$ where 
  $$ K_{\mf{p}}(p^{r_p}) = \Sp_4(\mb{Q}) \cap  \begin{bmatrix} 
    \mb{Z}_p & p^{r_p}\mb{Z}_p & \mb{Z}_p & \mb{Z}_p  \\
     \mb{Z}_p & \mb{Z}_p & \mb{Z}_p &  p^{-r}\mb{Z}_p \\
    \mb{Z}_p & p^{r_p}\mb{Z}_p & \mb{Z}_p &  \mb{Z}_p    \\ 
    p^{r_p}\mb{Z}_p & p^{r_p}\mb{Z}_p & p^{r_p}\mb{Z}_p & \mb{Z}_p     \\  \end{bmatrix} $$ \end{enumerate}  

For each cusp form, there is an associated automorphic representation over $\G_4(\mb{A_Q})$. For a fixed positive integer $N$,  let $F \in S_k(\Gamma^{(2)}(N))$ be a cusp form, $\Gamma^{(2)}(N)$ be any one of the four congruence subgroups defined above and $K_{\mf{f}}$ be the open compact group such that $\Gamma^{(2)}(N) = \G_4(\mb{Q}) \cap \G_4(\mb{R})^+K_{\mathfrak{f}}.$ The Strong Approximation Theorem for $\G(\mb{A_Q})$ states that $$\G_4(\mb{A_Q}) \cong \G_4(\mb{Q}) (\G_4(\mb{R})^+ K_{\mathfrak{f}}). $$ 
  
It implies that, given $g \in \G_4(\mb{A_Q})$ there exists $g_q \in \G_4(\mb{Q})$, $g_{\infty} \in \G_4(\mb{R})^+$, $k \in K_{\mathfrak{f}}$ such that
$ g = g_q(g_{\infty} k).$ An automorphic form associated with $F$ is a function on $\G_4(\mb{A_Q})$ denoted by $\phi_F$ and defined as follows: For $g \in \G_4(\mb{A_Q}),$ 
$$ \phi_F(g) \coloneqq \mu(g_{\infty})^k \mathrm{det}(C_{\infty} I_2 +D_{\infty})^{-k}F\big((A_{\infty} I_2 +B_{\infty})(C_{\infty} I_2 + D_{\infty})^{-1}\big) $$where $$I_2 = \begin{bmatrix}
    i & 0 \\ 0 & i 
\end{bmatrix} \quad \text{and} \quad g_{\infty}= \begin{bmatrix}
    A_{\infty} & B_{\infty} \\ C_{\infty} & D_{\infty} 
\end{bmatrix}.$$ From the fact that $F \in \mr{S}_k(\Gamma^{(2)}(N))$ and the strong approximation theorem, it follows that $\phi_F$ is well defined. It can also be shown that $\phi_F \in L^2(Z(\mb{A_Q})\G_4(\mb{Q}) \backslash \G_4(\mb{A_Q})).$ Given  $ h \in \G_4(\mb{A_Q})$, we define right translation of $\phi_F$ by 
$$ h.\phi_F(g) \coloneqq \phi_F(gh).$$ Let $V_F$ denote the subspace of $ L^2(Z(\mb{A_Q})\G_4(\mb{Q}) \backslash \G_4(\mb{A_Q}))$ generated by $h.\phi_F$ for $h \in \G_4(\mb{A_Q})$. The group $\G_4(\mb{A_Q})$ acts on $V_F$ by right translation. This action is defined as the representation associated with $F$ and is denoted by $\pi_F.$ More details on the construction of $\pi_F$ can be found in Section $4$ of \cite{MR1821182} and Section $3.2$ of \cite{MR2114732}. 

Since the representation is trivial on the center of $\G_4(\mb{A_Q})$, it can be seen as a representation of $\mr{PGSp}_4(\mb{A_Q}).$ Using the exceptional isomorphism, $\mr{PGSp}_4(\mb{A_Q}) \cong \mr{SO}_5(\mb{A_Q})$, $\pi_F$ can be extended to a representation of $\mr{SO}_5(\mb{A_Q})$. Hence, given $F$, we can attach a representation of $\mr{SO}_5(\mb{A_Q}).$ Theorem $1.3.2$ of \cite{MR3135650} gives a classification of all such representations. In Section $2.2$ of \cite{MR3766842}, Schmidt explains the classification specific to the case of modular forms of the genus $2$. In this case, there are $6$ distinct classes. In this article, we focus on modular forms that satisfy the generalized Ramanujan conjecture.  \\
\textbf{Generalized Ramanujan Conjecture:} Let $F \in S_k(\Gamma^{(n)}(N))$ be a Hecke eigenform with Satake-$p$-parameters $\alpha^{(F)}_{0,p},\alpha^{(F)}_{1,p},...,\alpha^{(F)}_{n,p}.$ A prime $p$ is called unramified if $p \nmid N$. GRC states that for all the unramified primes $p$, the Satake-$p$-parameters satisfy 
$$\lvert \alpha_{i,p} \rvert = 1 \ \text{for} \ i=1,2,...,n.$$ \cite[Prop 2.1]{MR3766842} proves that \textbf{G} and \textbf{Y} are the only classes that satisfy the Ramanujan conjecture.  
\begin{itemize}
    \item $General \ type, \ (\textbf{G}):$  $F \in S_k(\Gamma^{(2)}(N))$ is said to be of type \textbf{G}, if there exists a cuspidal automorphic representation $\pi$ of $\mr{GL}_4(\mb{A_Q})$ such that 
$$L(s,\pi_F,spin) = L(s,\pi).$$
\item $Yoshida \ type, \ (\textbf{Y}):$ $F \in S_k(\Gamma^{(2)}(N))$ is said to be of type \textbf{Y}, if there exists two cuspidal automorphic representations $\pi_1, \ \pi_2$ of $\mr{GL}_2(\mb{A_Q})$ such that $$L(s,\pi_F,spin) = L(s,\pi_1)L(s,\pi_2).$$ An example of such modular forms are the Yoshida lifts. 
\end{itemize}
\begin{definition}[Yoshida lifts] Let $f \in S_{k_1}(\Gamma_0(N_1))$ and $g \in S_{k_2}(\Gamma_0(N_2))$ be two Hecke eigen newforms where 
$$ \Gamma_0(N) = \{ \begin{pmatrix} 
    a & b  \\
    c & d  \\
  
    \end{pmatrix}\in \mr{SL}_2(\mb{Z}) : \ c \equiv 0 (\mr{mod}N) \}.$$ $F \in S_k(\Gamma^{(2)}(N))$ is said to be a Yoshida lift of $f$ and $g$, if $\pi_F$ is irreducible and $$L(s,\pi_F, \mr{spin}) = L(s,\pi_f)L(s,\pi_g).$$
\end{definition}

\section{First eigenvalue} 
For each $g \in \G_{2n}(\mathbb{Q})^{+} \cap M_{2n}(\mathbb{Z})$ such that $\mr{gcd}(\mu(g),N)=1$ we can associate a Hecke operator $T(g)$  on $M_k(\Gamma^{(n)}(N)).$ Let $\Gamma = \Sp_{2n}(\mb{Z})$, for $F\in M_k(\Gamma^{(n)}(N)),$

$$ \ T(g)F \coloneqq \sum_i F|_kg_i \quad \text{where} \  \Gamma^{(n)}(N) g \Gamma^{(n)}(N) = \sqcup_i \Gamma^{(n)}(N) g_i,  \quad g_i = \begin{pmatrix} 
    A_i & B_i  \\
    C_i & D_i  \\
  
    \end{pmatrix}$$
$$ \text{and} \ F|_kg_i(Z) = \mu (g)^{nk-\frac{n(n+1)}{2}}\mathrm{det}(C_iZ+D_i)^{-k}F((A_iZ+B_i)(C_iZ+D_i)^{-1}).$$ For a positive integer $m$ such that $\mr{gcd}(m,N)=1,$ 
$$T(m) \coloneqq \sum_{g : \mu(g) = m} T(g) .$$ In Theorem $4.7$ of \cite{MR2468862}, it is proved that there exists a basis for $M_k(\Gamma^{(n)}(N))$ which are eigenforms with respect to all Hecke operators $T(p)$ such that $p \nmid N$. For a Hecke eigenform $F \in \mr{S}_k(\Gamma^{(n)}(N))$, denote $\mu_F(g)$ as the eigenvalue of the operator $T(g)$. Classically $\mu_F(g)$ can be expressed in terms of $Satake \ p-parameters$. For any $g$ with $\mu (g) = p^r$, depending on $F$ there are $n+1$ complex numbers $(a^{(F)}_{0,p},a^{(F)}_{1,p},\dots,a^{(F)}_{n,p})$ satisfying 
\begin{equation} \mu_F(g) = {(p^{nk-\frac{n(n+1)}{4}}a^{(F)}_{0,p})}^r \sum_i \prod_{j=1}^n (a^{(F)}_{j,p}p^{-j})^{d_{ij}} \ \text{where} \ \Gamma^{(n)}(N) g \Gamma^{(n)}(N) = \sqcup_i \Gamma^{(n)}(N) g_i,\end{equation} 
$$
g_i = \begin{pmatrix} 
    A_i & B_i  \\
    0 & D_i  \\
  \end{pmatrix}
\ \text{and} \ 
D_i = \begin{pmatrix} 
    p^{d_{i1}} &    & \ast \\
      & \ddots & \vdots \\
    0 & \dots  & p^{d_{in}} 
    \end{pmatrix}.
 $$ The complex numbers $a^{(F)}_{j,p}$ for $0 \leq j \leq n$ are called the Satake $p$ parameters of $F$.  
 \begin{lemma} If $F \in \mr{S}_k(\Gamma^{(2)}(N))$ is a Hecke eigenform and $a^{(F)}_{0,p}, \ a^{(F)}_{1,p}, \ a^{(F)}_{2,p}$ are the Satake-p-parameters then 
 $$\mu_F(p) = p^{2k-\frac{3}{2}}(a^{(F)}_{0,p}+a^{(F)}_{0,p}a^{(F)}_{1,p}+a^{(F)}_{0,p}a^{(F)}_{2,p}+a^{(F)}_{0,p}a^{(F)}_{1,p}a^{(F)}_{2,p}).$$  \end{lemma} 
\begin{proof} For the Hecke operator $T(p)$ we have the following decomposition $$T(p) =  \Gamma_2 \begin{bmatrix} 
    1_2 &   \\
      & p1_2  \\
  \end{bmatrix} \Gamma_2 = \Gamma_2 \begin{bmatrix} 
    p1_2 &   \\
      & 1_2  \\
  \end{bmatrix} \sqcup \bigsqcup_{a \in \Z/p\Z} \Gamma_2  \begin{bmatrix} 
    1 & & a &   \\
      & p &  &   \\
     & & p&      \\ 
     & & & 1     \\  
  \end{bmatrix}\sqcup \bigsqcup_{\alpha,d \in \Z/p\Z} \Gamma_2  $$ $$  \begin{bmatrix} 
    p & &  &   \\
     -\alpha  & 1 &  &d   \\
     & & 1 & \alpha     \\ 
     & & & p     \\  
  \end{bmatrix}\sqcup \bigsqcup_{a,b,d \in \Z/p\Z} \Gamma_2  \begin{bmatrix} 
    1 & & a & b   \\
      & 1 & b &  d \\
     & & p&      \\ 
     & & & p     \\  
  \end{bmatrix}.$$There are four kinds of right cosets in the above decomposition,
  $$ g_1 = \begin{bmatrix} 
    p1_2 &   \\
      & 1_2  \\
  \end{bmatrix}, \ g_2 = \begin{bmatrix} 
    1 & & \ast &   \\
      & p &  &   \\
     & & p&      \\ 
     & & & 1     \\  
  \end{bmatrix},  \  g_3 = \begin{bmatrix} 
    p & &  &   \\
     \ast  & 1 &  &\ast   \\
     & & 1 & \ast     \\ 
     & & & p     \\  
  \end{bmatrix}, \ g_4 = \begin{bmatrix} 
    1 & & \ast & \ast   \\
      & 1 & \ast &  \ast \\
     & & p&      \\ 
     & & & p     \\  
  \end{bmatrix}. $$With $D_i'$s,
  $$D_1 = \begin{bmatrix} 
  1 & 0 \\
  0 & 1 \end{bmatrix}, \ D_2 = \begin{bmatrix} 
  p & 0 \\
  0 & 1 \end{bmatrix}, \ D_3 = \begin{bmatrix} 
  1 & \ast \\
  0 & p \end{bmatrix}, \ D_4 = \begin{bmatrix} 
  p & 0 \\
  0 & p \end{bmatrix}.$$ $\mu_F(p)$ can be calculated using the formula in $(1)$ and evaluating the contribution of each $D_i$. For $D_1$, $d_{1,1}=d_{1,2} = 0$. Hence the contribution to $\mu_F(p) $ is $p^{2k-\frac{3}{2}}a^{(F)}_{0,p}.$ For $D_2$, $d_{2,1} = 1$ and $d_{2,2}=0$, and there are $p$ number of cosets of these kind. Adding each cosets contribution to the eigenvalue we get $$ p^{2k-\frac{3}{2}}a^{(F)}_{0,p}pa_{1,p}^{(F)}p^{-1} = p^{2k-\frac{3}{2}}a^{(F)}_{0,p}a^{(F)}_{1,p}.$$ Similarly the contribution of $D_3'$s and $D_4'$s comes out to be $p^{2k-\frac{3}{2}}a^{(F)}_{0,p}a^{(F)}_{2,p}$ and $p^{2k-\frac{3}{2}}a^{(F)}_{0,p}a^{(F)}_{2,p}a^{(F)}_{2,p}$ respectively. Adding everything, it follows that 
  $$\mu_F(p) = p^{2k-\frac{3}{2}}( a^{(F)}_{0,p} + a^{(F)}_{0,p}a^{(F)}_{1,p}+a^{(F)}_{0,p}a^{(F)}_{2,p}+a^{(F)}_{0,p}a^{(F)}_{1,p}a^{(F)}_{2,p}).$$ \end{proof}
  Let $\pi_F$ be the automorphic representation associated with $F$ and $ L(s, \pi_F)$ be the corresponding spin L-function. $ \pi_F = \otimes_{p}^{'} \pi_p$ where $\pi_p$ is an unramified representation of $\G_{4}(\mb{Q}_p)$ for all primes $p \nmid N.$ And $\ds{ L(s, \pi_F) = \prod_{p} L_p(s,\pi_{F,p})}$, where $L_p(s, \pi_{F,p})$ are called the local L-factors. There exists complex numbers $a_{1,p}, a_{2,p}, a_{3,p}$ and $a_{4,p}$ such that $$ L_p(s,\pi_{F,p}) = \big((1-a_{1,p}p^{-s})(1-a_{2,p}p^{-s})(1-a_{3,p}p^{-s})(1-a_{4,p}p^{-s})\big)^{-1}.$$ For unramified primes, i.e for $p \nmid N$, these are the Satake-$p$-parameters associated to $F$. Hence
$$L_p(s,\pi_p) = \Big((1 - a^{(F)}_{0,p}p^{-s})(1 - a^{(F)}_{0,p}a^{(F)}_{1,p}p^{-s})(1 - a^{(F)}_{0,p}a^{(F)}_{2,p}p^{-s})(1 - a^{(F)}_{0,p}a^{(F)}_{1,p}a^{(F)}_{2,p}p^{-s})\Big)^{-1}$$ for all $ p \nmid N.$
  At ramified primes, the local L-factor is still an inverse of a polynomial in $p^{-s}$ but the degree can be less than $4$. Hence, we can write the local factor with $4$ constants $a_{i,p}$ but these can be zero as well. 
  Since we are interested in signs of $\mu_F(p)$, it is enough to study normalized eigenvalues $\lambda_{F}(p) = \frac{\mu_F(p)}{p^{2k-\frac{3}{2}}} $. From Lemma $2.1$, we conclude that $\lambda_F(p) = a_{1,p}+a_{2,p}+a_{3,p}+a_{4,p}$. If the local factors are written as Dirichlet series, say $L_p(s,\pi_F) = \sum_{r=1}^\infty a_{\pi_F}(p^r) p^{-rs} $, then at unramified primes $\lambda_F(p) = a_{\pi_F}(p)$. 
  
For a Hecke eigenform in class \textbf{G}, there exists an irreducible cuspidal automorphic representation $\pi$ of $\mr{GL}_4(\mb{A_Q})$ such that $\lambda_F(p) = a_{\pi}(p)$ for all unramified primes. Similarly, for a cusp form in class \textbf{Y}, there exists two irreducible cuspidal automorphic representations $\pi$ and $\tau$ of $\GL_2(\mb{A_Q})$ such that for all unramified primes $p$, $\lambda_F(p) = a_{\pi}(p)+a_{\tau}(p)$. 

In the remainder of the section we prove a few technical results that will be used for the main theorem. For any two real valued functions $f(x)$ and $g(x)$ we use the following notations. \begin{enumerate}
    \item $f(x) = O(g(x))$ if there exists a constant $c$ such that $ \lvert f(x) \rvert \leq c \lvert g(x) \rvert$ for sufficiently large $x$.
    \item $f(x) = o(g(x)$ if $\displaystyle{\lim_{x \rightarrow \infty} \frac{f(x)}{g(x)} = 0.}$
\end{enumerate} 
\begin{lemma} Let $\pi$ be a self dual, unitary,  cuspidal automorphic representation of $\GL_m(\mb{A_Q})$ for $m \leq 4$. If $L_p(s, \pi) = \sum_{r=1}^\infty a_{\pi}(p^r)p^{-rs} $ and $a_{\pi}(p)$ is bounded for all but finitely many primes then  $$\ds{\sum_{p \leq x} a_{\pi}(p)^2 = \frac{x}{\log{x}} + o\Big(\frac{x}{\log{x}} \Big) }$$\end{lemma} 
\begin{proof} Say $S$ is the finite set of primes such that $a_{\pi}$ is bounded for all $ p \not \in S.$ Applying \cite[Theorem 3]{MR2364718} for $\pi$ with $\tau_0 = 0$,  we get $$ \sum_{p \leq x} (\log{p}) a_{\pi}(p)^2 = x + O(xe^{-c \sqrt{\log{x}}}).$$This can be written as, 
$$\lim_{x \rightarrow \infty} \frac{\lvert \sum_{p \leq x} (\log{p}) a_{\pi}(p)^2 - x \rvert }{xe^{-c \sqrt{\log{x}}}} < \infty $$ 
Since $\ds{\lim_{x \rightarrow \infty} e^{c\sqrt{\log{x}}} = \infty }$, 
$$  \lim_{x \rightarrow \infty} \frac{\sum_{p \leq x} (\log{p}) a_{\pi}(p)^2 - x}{x} = 0.$$Hence,  
$$ \sum_{p \leq x} (\log{p}) a_{\pi}(p)^2 = x +o(x). $$It implies that 
$$  \sum_{p \leq x} (\log{p}) a_{\pi}(p)^2 - \sum_{p \leq x} (\log{x}) a_{\pi}(p)^2  + \sum_{p \leq x} (\log{x}) a_{\pi}(p)^2 = x +o(x) $$We note that, 
$$ \lim_{x \rightarrow \infty} \frac{\sum_{p \leq x} (\log{x}) a_{\pi}(p)^2 - \sum_{p \leq x} (\log{p}) a_{\pi}(p)^2}{x} = \lim_{x \rightarrow \infty} \sum_{p \leq x} \frac{\log{x}-\log{p}}{x}a_{\pi}(p)^2 $$ 
$$= \lim_{x \rightarrow \infty} \sum_{p \leq x, \ p \in S } \frac{\log{x}-\log{p}}{x}a_{\pi}(p)^2+\lim_{x \rightarrow \infty} \sum_{p \leq x, \ p \not \in S} \frac{\log{x}-\log{p}}{x}a_{\pi}(p)^2 .$$ Since the first limit has finite summation and $  \lim_{x \rightarrow \infty}\frac{\log{x}-\log{p}}{x} =0$, the first limit is $0$. Say $\lvert a_{\pi}(p) \rvert \leq M$ for $ p \not \in S$ then   
$$\lim_{x \rightarrow \infty} \sum_{p \leq x, \ p \not \in S} \frac{\log{x}-\log{p}}{x}a_{\pi}(p)^2 \leq M^2  \lim_{x \rightarrow \infty} \sum_{p \leq x} \frac{\log{x}-\log{p}}{x} .$$ $\ds{\sum_{p \leq x} \log{p}}$ is called first Chebyshev's function and it is denoted by $\vartheta(x)$. 
$$ \lim_{ x \rightarrow \infty}\sum_{p \leq x}  \frac{\log{x} - \log{p}}{x} = \lim_{ x \rightarrow \infty}\sum_{p \leq x}  \frac{\log{x}}{x} -  \lim_{ x \rightarrow \infty} \frac{\vartheta(x)}{x} = \lim_{ x \rightarrow \infty} \frac{\pi(x) \log{x}}{x} -  \lim_{ x \rightarrow \infty} \frac{\vartheta(x)}{x}.$$Prime number theorem states that, $\ds{ \lim_{ x \rightarrow \infty} \frac{\pi(x) \log{x}}{x} = \lim_{ x \rightarrow \infty} \frac{\vartheta(x)}{x} = 1}$. Hence 
$$  \lim_{x \rightarrow \infty} \sum_{p \leq x, \ p \not \in S} \frac{\log{x}-\log{p}}{x} =0 .$$ And we conclude that,
$$ \lim_{x \rightarrow \infty} \frac{\sum_{p \leq x} (\log{x}) a_{\pi}(p)^2 - \sum_{p \leq x} (\log{p}) a_{\pi}(p)^2}{x} = 0.$$Hence, $\ds{ \sum_{p \leq x} (\log{p}) a_{\pi}(p)^2 - \sum_{p \leq x} (\log{x}) a_{\pi}(p)^2 = o(x)}$ and
$$ \sum_{p \leq x} (\log{x}) a_{\pi}(p)^2 = x + o(x) $$ Dividing the above equation by $\log{x}$ on both sides proves the lemma.  \end{proof}  
\begin{lemma} Let $\pi_1$ and $\pi_2$ be cuspidal automorphic representations of $\GL_m(\mb{A_Q})$ for $m \leq 4.$ Assume that they have trivial central character, $ \pi_1 \not \cong \pi_2$ and there exists a finite set of primes $S$ such that $\lvert a_{\pi_1}(p)\ a_{\pi_2}(p)\rvert \leq M $ for some positive constant $M$ and for all $ p \not \in S.$ Then
$$\ds{\sum_{p \leq x}  a_{\pi_1}(p)a_{\pi_2}(p) = o\Big( \frac{x}{\log{x}} \Big)}.$$ 
\end{lemma} 
\begin{proof} Apply \cite[Theorem 3]{MR2364718} for $\pi_1$ and $\pi_2$ with $\tau_0 = 0$, we get $$ \sum_{p \leq x} (\log{p}) a_{\pi_1}(p)a_{\pi_2}(p) = O(xe^{-c \sqrt{\log{x}}}). $$
Since $\ds{\lim_{x \rightarrow \infty} e^{c\sqrt{\log{x}}} = \infty }$, similar to previous lemma we conclude that
$$ \sum_{p \leq x} (\log{p}) a_{\pi_1}(p)a_{\pi_2}(p) =o(x),$$ which can be written as
$$ \sum_{p \leq x} (\log{p}) a_{\pi_1}(p)a_{\pi_2}(p) - \sum_{p \leq x} (\log{x}) a_{\pi_1}(p)a_{\pi_2}(p)  + \sum_{p \leq x} (\log{x}) a_{\pi_1}(p)a_{\pi_2}(p) = o(x) .$$ We note that, 
$$ \lim_{x \rightarrow \infty} \frac{\sum_{p \leq x} (\log{x}) a_{\pi_1}(p)a_{\pi_2}(p) - \sum_{p \leq x} (\log{p}) a_{\pi_1}(p)a_{\pi_2}(p)}{x} = \lim_{x \rightarrow \infty} \sum_{ p \leq x} \frac{\log{x}-\log{p}}{x} a_{\pi_1}(p)a_{\pi_2}(p) $$ 
$$=\lim_{x \rightarrow \infty} \sum_{ p \leq x, \ p \in S} \frac{\log{x}-\log{p}}{x} a_{\pi_1}(p)a_{\pi_2}(p)+\lim_{x \rightarrow \infty} \sum_{ p \leq x, \ p \not \in S} \frac{\log{x}-\log{p}}{x} a_{\pi_1}(p)a_{\pi_2}(p).$$ Using the bound $\lvert a_{\pi_1}(p)a_{\pi_2}(p) \rvert \leq M$ for $ p \not \in S$, the above summation is   
$$ \leq \lim_{x \rightarrow \infty} \sum_{ p \leq x, \ p \in S} \frac{\log{x}-\log{p}}{x} a_{\pi_1}(p)a_{\pi_2}(p) + M  \lim_{x \rightarrow \infty} \sum_{p \leq x} \frac{\log{x}-\log{p}}{x} = 0 .$$ Similar to the previous lemma, the first limit is zero since the summation is a finite sum and the second limit is zero using the Prime number theorem. Hence, $$\sum_{p \leq x} (\log{p}) a_{\pi_1}(p)a_{\pi_2}(p) - \sum_{p \leq x} (\log{x}) a_{\pi_1}(p)a_{\pi_2}(p)= o(x)$$ and 
$$ \sum_{p \leq x} (\log{x}) a_{\pi_1}(p)a_{\pi_2}(p) =  o(x). $$  \end{proof}
\begin{corollary} 
Let $F \in \mr{S}_k(\Gamma^{(2)}(N))$ be a normalized Hecke eigenform for all primes $ p \nmid N$ and satisfies the Ramanujan conjecture. Let $\lambda_F(p)$ represent the eigenvalue for the operator $T(p)$ for all $ p \nmid N$. Then 
$$ \sum_{ p \leq x, \ p \nmid N } \lambda_F(p)^2 =m \frac{x}{\log{x}} + o\Big( \frac{x}{\log{x}} \Big)  $$ where 
$$m=
    \begin{cases}
        1 & \text{if } F \in \textbf{G}\\
        2 & \text{if } F \in \textbf{Y}.
    \end{cases}$$

\end{corollary} 
\begin{proof} Let $\pi$ be a self dual unitary cuspidal automorphic representation over $\GL_m(\mb{Q})$ for $m \leq 4.$ Lemma $2.2$ can be written as 
$$ \sum_{ p\leq x, \ p \mid N} a_{\pi}(p)^2+\sum_{ p\leq x,\ p \nmid N} a_{\pi}(p)^2 = \frac{x}{\log{x}} + o\Big( \frac{x}{\log{x}} \Big) .$$ The finite sum can be absorbed into $\ds{o\Big( \frac{x}{\log{x}} \Big)}$ and we conclude $$\sum_{ p\leq x,\ p \nmid N} a_{\pi}(p)^2 = \frac{x}{\log{x}} + o\Big( \frac{x}{\log{x}} \Big).$$

If $F$ is in class \textbf{G} then there exists a self dual, unitary cuspidal automorphic representation $\pi$ of $\GL_4(\mb{Q})$ such that $\lambda_F(p) = a_{\pi}(p)$ for all $ p \nmid N.$ Hence, the corollary follows from Lemma $2.2.$ If $F$ is in class \textbf{Y} then $\lambda_F(p) = a_{\pi_1}(p)+a_{\pi_2}(p)$ for all $ p \nmid N$ where $\pi_1,\pi_2$ are distinct self dual, unitary cuspidal automorphic representations of $\GL_2(\mb{Q}).$ In this case,
$$ \lambda_F(p)^2 = a_{\pi_1}(p)^2+a_{\pi_2}(p)^2 + 2a_{\pi_1}a_{\pi_2}.$$ Applying Lemma $2.2$ and $2.3$ we get, 
$$ \sum_{p \leq x, \ p \nmid N}\lambda_F(p)^2 =  \frac{x}{\log{x}} + o\Big( \frac{x}{\log{x}} \Big)+\frac{x}{\log{x}} + o\Big( \frac{x}{\log{x}} \Big) + o\Big( \frac{x}{\log{x}} \Big)= 2 \frac{x}{\log{x}} + o\Big( \frac{x}{\log{x}} \Big) .$$\end{proof}

\section{Main result}
The main result of this article is on comparing the signs of eigenvalues of two distinct modular forms satisfying the Ramanujan conjecture. 

\begin{lemma} Let $F\in \mr{S}_{k_1}(\Gamma^{(2)}(N_1))$ and $G \in \mr{S}_{k_2}(\Gamma^{(2)}(N_2))$ be two Hecke eigenforms satisfying the Ramanujan conjecture. Assume that, for some $c \in (0,4)$ and $ \alpha > \frac{15}{16}$, $ \# \{ p \leq x : \lvert \lambda_G(p) \rvert > c \} \geq \alpha \frac{x}{\log{x}} $ for sufficiently large $x$. Let $S$ contain all the primes $p$ dividing $N_1$ and $N_2$. Then, $$ \sum_{p \leq x, \ p \not \in S} \lambda_F(p)^2\lambda_G(p)^2 \geq \beta \frac{x}{\log{x}} + o\Big( \frac{x}{\log{x}}\Big)$$ for some $\beta \in (0,1).$
\end{lemma}  
 \begin{proof} Prime Number Theorem states that $\ds{\# \{ p \leq x  : p \ \text{is prime} \}  = \frac{x}{\log{x}} + o\Big( \frac{x}{\log{x}} \Big)}$. Since removing finitely many primes would not effect the asymptotic behavior, we conclude that $$ \# \{ p \leq x  : p \ \text{is prime and} \ p \not \in S  \}  = \frac{x}{\log{x}} + o\Big( \frac{x}{\log{x}} \Big).$$ Say $S_g(x) = \{ p \leq x : \lvert \lambda_G(p) \rvert > c \}$. The above equation can be written as, 
$$ \# \{ p \leq x : p \not \in S_g(x) \ \text{and} \ p \not \in S \} = \frac{x}{\log{x}} + o\Big( \frac{x}{\log{x}} \Big) - \# \{ p \leq x : p \in S_g(x) \ \text{and} \ p \not \in S \} .$$ Under the assumption, $ \ds{ \# S_g(x) \geq \alpha \frac{x}{\log{x}} }$ we get, 
$$  \# \{ p \leq x : p \not \in S_g \ \text{and} \ p \not \in S \} \leq (1-\alpha) \frac{x}{\log{x}} + o\Big( \frac{x}{\log{x}} \Big) $$for sufficiently large $x$. By Weissauer's bound proved in \cite{MR2498783}, $ \lvert \lambda_F(p) \rvert  \leq 4 $ for all primes $p \not \in S.$
$$ \sum_{ p \leq x, \ p \not \in S_g(x) } \lambda_F(p)^2 \leq 4^2 \# \{ p \leq x: p \not \in S \cup S_g(x) \} \leq  16(1-\alpha) \frac{x}{\log{x}} + o \Big( \frac{x}{\log{x}} \Big). $$Combining it with Corollary $2.3.1$ we get,  
$$ \sum_{p \leq x, \ p \in S_g(x)} \lambda_F(p)^2 \geq (16\alpha+m-16) \frac{x}{\log{x}} + o \Big( \frac{x}{\log{x}} \Big). $$ 
$$ \sum_{ p \leq x, \ p \not \in S} \lambda_F(p)^2 \lambda_G(p)^ 2 \geq  \sum_{ p \leq x, \ p \in S_g(x) } \lambda_F(p)^2 \lambda_G(p)^ 2 \geq c^2(16 \alpha +m-16) \frac{x}{\log{x}} + o \Big( \frac{x}{\log{x}} \Big). $$This proves that \begin{equation}  
 \sum_{ p \leq x, \ p \not \in S} \lambda_F(p)^2 \lambda_G(p)^ 2 \geq   c^2(16 \alpha+m-16) \frac{x}{\log{x}} + o \Big( \frac{x}{\log{x}} \Big). \end{equation} 
\end{proof}
 \begin{lemma} Let $F\in \mr{S}_{k_1}(\Gamma^{(2)}(N_1))$ and $G \in \mr{S}_{k_2}(\Gamma^{(2)}(N_2))$ be two normalised Hecke eigenforms satisfying the Ramanujan conjecture. Assume that if both $F,G$ lift to class \textbf{Y} then $$L(s,\pi_F,\mr{spin}) = L(s,\pi_1)L(s,\pi_2) \ \text{and} \ L(s,\pi_G,\mr{spin}) = L(s,\tau_1)L(s,\tau_2) $$ where $\pi_1, \pi_2, \tau_1$ and $\tau_2$ are all distinct automorphic representation over $\GL_2(\mb{A_Q}).$ Under these assumptions, if $\lambda_F(p)$ and $\lambda_G(p)$ are eigenvalues of $F$ and $G$ respectively then \begin{equation} 
\sum_{p \leq x, \ p \nmid N} \lambda_F(p) \lambda_G(p) = o \Big( \frac{x}{\log{x}} \Big).  \end{equation}
 \end{lemma}
 \begin{proof}
 Let $\pi_F$, $\pi_G$ be automorphic representations associated with $F$ and $G$ respectively. There are $4$ different possibilities for their L-functions depending on the class of lifts. \begin{enumerate} 
\item $L(s,\pi_F,\mr{spin}) = L(s,\pi_1)$ and $L(s,\pi_G,\mr{spin}) = L(s,\pi_2)$ such that $\pi_1,\pi_2$ are distinct self dual, unitary cuspidal automorphic representations of $\GL_4(\mb{A_Q}).$  
\item $L(s,\pi_F,\mr{spin}) = L(s,\pi)$ such that $\pi$ is self dual, unitary cuspidal automorphic representations of $\GL_4(\mb{A_Q})$. $L(s,\pi_G,\mr{spin})= L(s,\tau_1)L(s,\tau_2)$ where $\tau_1,\tau_2$ are distinct self dual cuspidal automorphic representations of $\GL_2(\mb{A_Q}).$ 
\item $L(s,\pi_F,\mr{spin})= L(s,\pi_1)L(s,\pi_2)$ and $L(s,\pi_G,\mr{spin})= L(s,\tau_1)L(s,\tau_2)$ where $\pi_1,\pi_2,\tau_1$ and $\tau_2$ are distinct self dual cuspidal automorphic representations of $\GL_2(\mb{A_Q}).$
\item $L(s,\pi_F,\mr{spin})= L(s,\pi)L(s,\tau_1)$ and $L(s,\pi_G,\mr{spin})= L(s,\pi)L(s,\tau_2)$ where $\pi$ is self dual cuspidal automorphic representations of $\GL_2(\mb{A_Q})$. $\tau_1$ and $\tau_2$ are distinct self dual cuspidal automorphic representations of $\GL_2(\mb{A_Q}).$
\end{enumerate} Assumptions imply that the fourth case is not possible. Hence,  \begin{enumerate}
     \item $\lambda_F(p) = a_{\pi_1}(p)$ and $\lambda_G(p) = a_{\pi_2}(p)$. Observe that $a_{\pi_1}$ and $a_{\pi_2}$ satisfy the conditions for Lemma $2.3.$ Hence, 
     $$ \sum_{p \nmid N, p \leq x} \lambda_F(p)\lambda_G(p) = \sum_{p \nmid N, p \leq x} a_{\pi_1}(p)a_{\pi_2}(p) = o \Big( \frac{x}{\log{x}} \Big).$$ 
     \item If $\lambda_F(p) = a_{\pi}(p)$ and $\lambda_G(p) = a_{\tau_1}(p)+a_{\tau_2}(p),$ then 
     $$ \sum_{p \nmid N, p \leq x} \lambda_F(p)\lambda_G(p) = \sum_{p \nmid N, p \leq x} a_{\pi}(p)a_{\tau_1}(p)+\sum_{p \nmid N, p \leq x} a_{\pi}(p)a_{\tau_2}(p) = o \Big( \frac{x}{\log{x}} \Big).$$
     \item If $\lambda_F(p) = a_{\pi_1}(p)+a_{\pi_2}(p).$ and $\lambda_G(p) = a_{\tau_1}(p)+a_{\tau_2}(p),$ then $ \sum_{p \nmid N, p \leq x} \lambda_F(p)\lambda_G(p) $
     $$= \sum_{p \nmid N, p \leq x} a_{\pi_1}(p)a_{\tau_1}(p)+\sum_{p \nmid N, p \leq x} a_{\pi_1}(p)a_{\tau_2}(p) +\sum_{p \nmid N, p \leq x} a_{\pi_2}(p)a_{\tau_1}(p)+\sum_{p \nmid N, p \leq x} a_{\pi_2}(p)a_{\tau_2}(p) $$ $$ = o \Big( \frac{x}{\log{x}} \Big) .$$
     
 \end{enumerate}
 \end{proof}

\begin{theorem} Let $F,G$ be two cusp forms satisfying the conditions of Lemma $3.2.$ Assume there exists a $c \in (0,4)$ and $\alpha > \frac{15}{16} $ such that 
$$\# \{ p \leq x : \lvert \lambda_{G}(p) \rvert > c \} \geq \alpha \frac{x}{\log{x}}$$ for sufficiently large $x$. Then, the set of primes $ \{ p : \lambda_{F}(p) \lambda_{G}(p) < 0 \}$, has positive density.
\end{theorem} 
\begin{proof}Let $S = \{ p : p \nmid N \}$. Consider the sum 
$$S^{-}(x) =  \sum_{p \leq x, \ p \not \in S} ( \lambda_F(p)^2\lambda_G(p)^2 - 16\lambda_F(p)\lambda_G(p)) =  \sum_{p \leq x, \ p \not \in S}  \lambda_F(p)\lambda_G(p)[\lambda_F(p)\lambda_G(p)-16] .$$ For $p \not \in S$, $\lvert \lambda_F(p) \lambda_G(p) \rvert \leq 16.$ Hence, for $p$ such that $\lambda_F(p)\lambda_G(p) > 0$, $\lambda_F(p)\lambda_G(p) - 16 < 0$. Therefore, $$S^{-}(x)   \leq  \sum_{p \leq x,\not \in S \ \lambda_F(p)\lambda_G(p) < 0} ( \lambda_F(p)^2\lambda_G(p)^2 - 16\lambda_F(p)\lambda_G(p)) $$ 
$$ \leq 512. \# \{ p \leq x  : p \not \in S \ \text{and} \ \lambda_F(p)\lambda_G(p) < 0 \}.$$ From Lemma $3.1$ and $3.2$, we conclude that 
$$ S^{-}(x) =  \sum_{p \leq x, \ p\not \in S} ( \lambda_F(p)^2\lambda_G(p)^2 - 16\lambda_F(p)\lambda_G(p)) \geq c^2(16 \alpha+m-16) \frac{x}{\log{x}} + o \Big( \frac{x}{\log{x}} \Big).$$ Combining the inequalities, 
$$  \# \{ p \leq x : p \not \in S \ \text{and} \  \lambda_F(p)\lambda_G(p) < 0 \} \geq \frac{c^2(16 \alpha+m-16)}{512}  \frac{x}{\log{x}} + o \Big( \frac{x}{\log{x}} \Big) .$$ Since $S$ contains finitely many primes, we can add them to the set to conclude  $$ \# \{ p \leq x : \lambda_F(p)\lambda_G(p) < 0 \} \geq \frac{c^2(16 \alpha+m-16)}{512} \frac{x}{\log{x}} + o\Big( \frac{x}{\log{x}} \Big).$$
Hence for $\alpha > \frac{15}{16}$ the set of primes $\{ p : \lambda_F(p)\lambda_G(p) < 0 \}$ has positive density. \end{proof}


\begin{thebibliography}{10}

\bibitem{addanki2024}
Nagarjuna~Chary Addanki.
\newblock On signs of {H}ecke eigenvalues of {I}keda lifts.
\newblock {\em arXiv \ 2401.08855}, 2024.

\bibitem{MR2468862}
Anatoli Andrianov.
\newblock {\em Introduction to {S}iegel modular forms and {D}irichlet series}.
\newblock Universitext. Springer, New York, 2009.

\bibitem{MR3135650}
James Arthur.
\newblock {\em The endoscopic classification of representations}, volume~61 of
  {\em American Mathematical Society Colloquium Publications}.
\newblock American Mathematical Society, Providence, RI, 2013.
\newblock Orthogonal and symplectic groups.

\bibitem{MR1821182}
Mahdi Asgari and Ralf Schmidt.
\newblock Siegel modular forms and representations.
\newblock {\em Manuscripta Math.}, 104(2):173--200, 2001.

\bibitem{MR1719682}
Stefan Breulmann.
\newblock On {H}ecke eigenforms in the {M}aa\ss space.
\newblock {\em Math. Z.}, 232(3):527--530, 1999.

\bibitem{MR4198744}
Sanoli Gun, Winfried Kohnen, and Biplab Paul.
\newblock Arithmetic behaviour of {H}ecke eigenvalues of {S}iegel cusp forms of
  degree two.
\newblock {\em Ramanujan J.}, 54(1):43--62, 2021.

\bibitem{MR2262899}
Winfried Kohnen.
\newblock Sign changes of {H}ecke eigenvalues of {S}iegel cusp forms of genus
  two.
\newblock {\em Proc. Amer. Math. Soc.}, 135(4):997--999, 2007.

\bibitem{MR2726725}
E.~Kowalski, Y.-K. Lau, K.~Soundararajan, and J.~Wu.
\newblock On modular signs.
\newblock {\em Math. Proc. Cambridge Philos. Soc.}, 149(3):389--411, 2010.

\bibitem{MR0549401}
Hans Maass.
\newblock \"uber eine {S}pezialschar von {M}odulformen zweiten {G}rades. {III}.
\newblock {\em Invent. Math.}, 53(3):255--265, 1979.

\bibitem{MR2425722}
Ameya Pitale and Ralf Schmidt.
\newblock Sign changes of {H}ecke eigenvalues of {S}iegel cusp forms of degree
  2.
\newblock {\em Proc. Amer. Math. Soc.}, 136(11):3831--3838, 2008.

\bibitem{MR2114732}
Ralf Schmidt.
\newblock Iwahori-spherical representations of {${\rm GSp}(4)$} and {S}iegel
  modular forms of degree 2 with square-free level.
\newblock {\em J. Math. Soc. Japan}, 57(1):259--293, 2005.

\bibitem{MR3766842}
Ralf Schmidt.
\newblock Packet structure and paramodular forms.
\newblock {\em Trans. Amer. Math. Soc.}, 370(5):3085--3112, 2018.

\bibitem{MR2498783}
Rainer Weissauer.
\newblock {\em Endoscopy for {${\rm GSp}(4)$} and the cohomology of {S}iegel
  modular threefolds}, volume 1968 of {\em Lecture Notes in Mathematics}.
\newblock Springer-Verlag, Berlin, 2009.

\bibitem{MR2364718}
Jie Wu and Yangbo Ye.
\newblock Hypothesis {H} and the prime number theorem for automorphic
  representations.
\newblock {\em Funct. Approx. Comment. Math.}, 37:461--471, 2007.

\end{thebibliography}
\end{document}